\documentclass[11pt]{article}
\usepackage{mathrsfs}
\usepackage{amsfonts,amsmath,dsfont,amssymb,amsthm,stmaryrd,bbm,color}
\usepackage[hidelinks]{hyperref}
\usepackage{bbm}

\newtheorem{conj}{Conjecture}[section]
\newtheorem{thm}{Theorem}[section]

\newtheorem{remark}[conj]{Remark}
\newtheorem{lem}[conj]{Lemma}
\newtheorem{prop}[conj]{Proposition}
\newtheorem{coro}[conj]{Corollary}

\newtheorem{defn}[conj]{Definition}
\newtheorem{cor}[conj]{Corollary}

\newcommand{\supp}{\mathrm{supp}}

\newcommand{\mR}{\mathbb{R}} 
\newcommand{\R}{\mathbb{R}} 
\newcommand{\C}{\mathbb{C}}

\newcommand{\mP}{\mathbb{P}}

\newcommand{\mE}{\mathbb{E}}

\newcommand{\re}{\mathrm{Re}\:}

\DeclareMathOperator{\Var}{Var}

\newcommand{\vol}{\mathrm{Vol}}

\newcommand{\h}{\widetilde{h}}
\renewcommand{\P}{\mathbb{P}}
\newcommand{\M}{\mathcal{M}}

\begin{document}

\title{Concentration of information content for convex measures\thanks{Some of  these results were announced at ISIT 2016 \cite{LFM16:isit}.}}
\author{Matthieu Fradelizi\thanks{LAMA, Univ Gustave Eiffel, UPEM, Univ Paris Est Cr\'eteil, CNRS, F-77447, Marne-la-Vall\'ee, France.
E-mail: matthieu.fradelizi@univ-eiffel.fr}, 
Jiange Li \thanks{Department of Mathematical Sciences, University of Delaware, 501 Ewing Hall, Newark, DE 19716, USA.
E-mail: lijiange@udel.edu}, 
Mokshay Madiman\thanks{Department of Mathematical Sciences, University of Delaware, 501 Ewing Hall, Newark, DE 19716, USA.
E-mail: madiman@udel.edu} 
}
\date{\today}
\maketitle

\begin{abstract}
We establish sharp exponential deviation estimates of the information content as well
as a sharp bound on the varentropy for the class of convex measures on Euclidean spaces.
This generalizes a similar development for log-concave measures in the recent work of Fradelizi, Madiman and Wang (2016).
In particular, our results imply that convex measures in high dimension are concentrated in an annulus
between two convex sets (as in the log-concave case) despite their possibly having much heavier tails.
Various tools and consequences are developed, including a sharp comparison result for R\'enyi entropies,
inequalities of Kahane-Khinchine type for convex measures that extend 
those of Koldobsky, Pajor and Yaskin (2008) for log-concave measures,
and an extension of Berwald's inequality (1947). 
\end{abstract}

\section{Introduction}

Let $X$ be a random vector in $\mR^n$. Suppose that the distribution of $X$ has 
density $f$ with respect to the Lebesgue measure on $\mR^n$. We say that  the random variable
\begin{align}
\widetilde{h}(X)=-\log f(X)
\end{align}
is the {\it information content} of $X$. (Throughout this paper, we denote by $\log$ the natural logarithm).
The average value of $\widetilde{h}(X)$ is known more commonly as the {\it entropy} $h(X)$. In other words, the entropy of $X$ is defined by
\begin{align}
h(X)=\mE\widetilde{h}(X)=-\int_{\mR^n}f(x)\log f(x)dx.
\end{align}

Because of the relevance of the information content in various areas such as information theory, probability and statistics, it is intrinsically interesting to understand its behavior. 
In particular, it is natural to ask whether the information content concentrates around the entropy in high dimension. 
If $X$ is a standard Gaussian random vector in $\mR^n$, its information content is
$$
\widetilde{h}(X)=\frac{|X|^2}{2}+\frac{n}{2}\log(2\pi),
$$
where $|\cdot|$ is the Euclidean norm. In this case, the concentration property of $\widetilde{h}(X)$
(or, equivalently, of $|X|^2$) around its mean is easy to show; the explicit computation was done, for example,
by Cover and Pombra \cite{CP89}, who were motivated by applications in communication theory. 
The first significant generalization beyond the Gaussian case was established by 
Bobkov and Madiman \cite{BM11:aop}, who showed that $\widetilde{h}(X)$ possesses a powerful concentration property 
if $X$ has a log-concave density. Specifically, they showed that there is a universal constant $c>1/16$ such that
for every random vector $X$ drawn from a log-concave density on $\mR^n$,
\begin{equation}\label{eq:bm-log-concave}
\P\left(\big|\h(X)-h(X)\big|\geq nt\right) \leq 2\,e^{-ct\sqrt{n}} .
\end{equation}
The proof of \cite{BM11:aop} heavily depends on the localization lemma of 
Lov\'asz-Simonovits \cite{LS93} and reverse H\"{o}lder-type inequalities \cite{Bor73a}. 
Fradelizi, Madiman and Wang \cite{FMW16} both improved this inequality (making
it sharp in a certain sense) and significantly simplified the proof, eliminating the use
of the localization lemma and instead only using the well known Pr\'ekopa-Leindler inequality
together with a new bootstrapping tool that they developed to deduce concentration bounds 
from certain uniform variance bounds.
Various applications have been found in recent years of this ``concentration of information'' phenomenon for log-concave measures 
(see, e.g., \cite{BM11:cras, BM12:allerton, BM12:jfa, BM13:goetze, MM14, MKC15:isit, DGP18, LG17, Kos17, MW17:witmse, MW19, MMX17:0, MK18, LMNPT18, MNT18}).

In this note, we extend the concentration property of the information content from log-concave measures to the significantly more 
general class of convex measures, which can have arbitrarily heavy tails. More precisely, for $s\in(-1/n, 0]$, we show that for every random vector $X$ drawn from a $s$-concave density on $\mR^n$ (see Section \ref{sec:moments} for definitions),
\begin{equation}
\mP\left(\widetilde{h}(X)-h(X)>nt\right)  \leq \exp\left(-\frac{n(1+ns)^2}{15}\min\{t, (1+ns)t^2\}\right) \label{eq:tail+}
\end{equation}
\begin{equation}
\mP\left(\widetilde{h}(X)-h(X)<-nt\right)\leq \exp\left(-\frac{n(1+ns)^3}{2}t^2\right).
\end{equation}
The limiting case $s=0$ recovers and strengthens the deviation estimate \eqref{eq:bm-log-concave}. Somewhat surprisingly, the functional that concentrates stays the same
and the only change is in the probability bounds--until this work, it was not clear whether one would have to consider functionals
of the density other than the information content in order to get meaningful concentration. Less surprisingly, the bounds (when appropriately
normalized) are no longer dimension-free as in the case of log-concave measures; this is tied to the folklore fact that one cannot have
dimension-free concentration in the absence of exponential tails. 

Various interesting consequences follow from our sharp exponential deviation estimate of the information content. For example, we have the following sharp comparison result for R\'enyi entropies of log-concave densities. Let $f$  be a log-concave density on $\mR^n$ and let $\phi$ be the one-sided exponential density on $\mR^n$, i.e., $\phi(x)=e^{-\sum_{i=1}^n x_i}$ when each $x_i>0$, and $\phi(x)=0$ otherwise. For $0<p<q\leq \infty$, we have
\begin{equation}
h_p(f)-h_q(f)\leq h_p(\phi)-h_q(\phi).
\end{equation}
Here, $h_p(f)$ is the R\'enyi entropy of order $p$ of the density $f$ (see Section \ref{sec:conseq} for definitions). This result was obtained in unpublished work of Madiman and Wang \cite{MW19}, 
but the role of exponential densities as maximizers is first explicitly stated here. Extensions of this result to the class of $s$-concave densities are given in Corollary \ref{cor:inf-h}. We also recover a result of Borell \cite{Bor73b} (see, p. 435, also Theorem 5.1 of \cite{GZ98}), which generalizes the classical theorem of Berwald \cite{Ber47}, see also \cite{MP89}, which was restricted to the range $p>0$. This is our Corollary \ref{cor:monot-f-K}, which states that for any concave function $f:K\to[0,\infty)$ on a convex body $K$ of $\mR^n$, the function 
$$
p\mapsto  \left( {n+p \choose n}\frac{1}{|K|}\int_K f(x)^pdx \right)^{1/p}
$$
is non-increasing on $(-1,\infty)$.

The rest of this note is organized as follows.
In Section~\ref{sec:gen}, we will show that exponential deviation of a functional follows from the log-concavity of normalized Laplace transform 
of that functional. In particular, we study the log-concavity of normalized moments of $s$-concave functions 
in Section~\ref{sec:moments}, and present a sharp result that unifies and extends results of
\cite{Bor73a, BM11:it, FMW16, BGG18, Ngu14:1}. 
Optimal concentration bounds and a sharp variance bound of the information content of 
$\kappa$-concave random vectors are established in Section~\ref{sec:opt} by combining the results
of Sections~\ref{sec:gen} and \ref{sec:moments}. These optimal bounds are put into
a more usable form in Section~\ref{sec:usable}, elucidating in particular the dependence of the bounds on dimension. In Section~\ref{sec:comp}, we examine a related monotonicity property of normalized moments of $s$-concave functions,
giving in particular an extension of a result of \cite{KPY08}.
Finally, various consequences of the main results of this note are discussed in Section~\ref{sec:conseq}.

\section{A general principle for exponential deviation}
\label{sec:gen}
 
A classical tool for establishing exponential deviation bounds of a random variable is the Cram\'er-Chernoff method, which relies on the control of Laplace transform of that random variable. In this section, we show that this could follow from the log-concavity of normalized Laplace transform, which is equivalent to establishing uniform variance bounds for tilted random variables. This was first developed by Fradelizi, Madiman and Wang \cite{FMW16} and is set in a more general framework here.

Let $X$ be a random vector in $\mR^n$. Let $\varphi: \mR^n \to\mR$ be a real-valued function. The {\it logarithmic Laplace transform} of $\varphi(X)$ is defined as
\begin{align}
L(\alpha)=\log\mE e^{\alpha\varphi(X)}.
\end{align}
The following observation is a well known fact about exponential families in statistics. 

\begin{lem} \label{lem:exp-family}
Suppose that $L(\alpha)<\infty$ for $\alpha\in(-a, b)$, where $a, b>0$ are certain real numbers. Then we have
$$
L'(\alpha)=\mE\varphi(X_\alpha), \ L''(\alpha)=\Var(\varphi(X_\alpha)),
$$
where $X_\alpha$ is the tilted random vector with density
$$
X_\alpha\sim \frac{e^{\alpha\varphi(x)}f(x)}{\int_{\mR^n}e^{\alpha\varphi(y)}f(y)dy}.
$$
In particular, we have $L'(0)=\mE\varphi(X)$ and $L''(0)=\Var(\varphi(X))$.
\end{lem}

\begin{proof}
The assumption $L(\alpha)<\infty$ for $\alpha\in(-a, b)$ guarantees that $L(\alpha)$ is infinitely differentiable with respect to $\alpha\in(-a, b)$ and that we can freely change the order of differentiation and expectation. Then we have
$$
L'(\alpha)=\frac{\int_{\mR^n} e^{\alpha\varphi(x)}f(x)\varphi(x)dx}{\int_{\mR^n} e^{\alpha\varphi(x)}f(x)dx}=\mE\varphi(X_\alpha).
$$
Differentiate $L'(\alpha)$ one more time. We have
$$
L''(\alpha)=\frac{\int_{\mR^n} e^{\alpha\varphi(x)}f(x)\varphi^2(x)dx}{\int_{\mR^n} e^{\alpha\varphi(x)}f(x)dx}-\left(\frac{\int_{\mR^n} e^{\alpha\varphi(x)}f(x)\varphi(x)dx}{\int_{\mR^n} e^{\alpha\varphi(x)}f(x)dx}\right)^2=\text{Var}(\varphi(X_\alpha)).
$$
\end{proof}

\begin{defn}\label{defn:log-concave}
A function $f: \mR^n\to [0, \infty)$ is called log-concave if we have
$$
f((1-\lambda)x+\lambda y)\geq f(x)^{1-\lambda}f(y)^\lambda
$$
for all $x,y\in\mR^n$ and all $\lambda\in[0,1]$.  
\end{defn}

\begin{lem}\label{lem:mgf-bound}
Suppose that $L(\alpha)<\infty$ for $\alpha\in(-a, b)$, where $a, b>0$ are certain real numbers. Let $c(\alpha)$ be a smooth function such that $e^{-c(\alpha)}\mE e^{\alpha\varphi(X)}$ is log-concave for $\alpha\in(-a, b)$. Then, for $\alpha\in(-a, b)$, we have
$$
\mE e^{\alpha(\varphi(X)-\mE\varphi(X))}\leq e^{\psi(\alpha)},
$$
where $\psi(\alpha)=c(\alpha)-c(0)-c'(0)\alpha$.
\end{lem}

\begin{proof}
Since $e^{-c(s)}\mE e^{s\varphi(X)}$ is log-concave, we have $L''(s)\leq c''(s)$. For any $0<t<\alpha< b$, integrating the previous inequality over $(0, t)$ we have
$$
L'(t)-L'(0)\leq c'(t)-c'(0).
$$
Integrating both sides of this inequality over $(0, \alpha)$, we have
\begin{align}\label{eq:log-moment}
L(\alpha)-L(0)-L'(0)\alpha\leq c(\alpha)-c(0)-c'(0)\alpha.
\end{align}
Similarly, one can show that the estimate also holds for $-a<\alpha<0$. Notice that $L(0)=0$ and $L'(0)=\mE \varphi(X)$. Then the lemma follows from exponentiating both sides of \eqref{eq:log-moment}.
\end{proof}

\begin{remark}
From Lemmas \ref{lem:exp-family} and \ref{lem:mgf-bound}, we can see that the variance bound of $\varphi(X_\alpha)$ and the normalizing function for the Laplace transform $\mE e^{\alpha\varphi(X)}$ to be log-concave can be obtained from each other by differentiating or integrating twice. 
\end{remark}

Now, we can apply the Cram\'er-Chernoff argument to establish exponential deviation inequalities. First, we introduce some notations. Let $f: \mR\to \mR\cup\{\infty\}$ be an extended real-valued function. The {\it Legendre transform} $f^*$ is defined as
\begin{align}
f^*(x)=\sup_{y\in\mR}(xy-f(y)), \ x\in\mR.
\end{align}
Define $f_+(x)=f(x){\bf 1}_{(0, \infty)}(x)$ and $f_-(x)=f(x){\bf1}_{(-\infty, 0)}(x)$, i.e., the restrictions of $f$ on the positive and negative half axes, respectively. 

\begin{coro}
Under assumptions and notations of Lemma \ref{lem:mgf-bound},  for any $t>0$, we have 
$$
\mP(\varphi(X)-\mE\varphi(X)>t) \leq e^{-\psi_+^*(t)}
$$
$$
\mP(\varphi(X)-\mE\varphi(X)<-t) \leq e^{-\psi_-^*(-t)},
$$
where $\psi_+^*$ and $\psi_-^*$ are Legendre transforms of $\psi_+$ and $\psi_-$, respectively.
\end{coro}

\begin{proof}
For the upper tail, we have for $0<\alpha<b$ and $t>0$ that
\begin{align*}
\mP(\varphi(X)-\mE(\varphi(X))>t) &= \mP\left(e^{\alpha(\varphi(X)-\mE\varphi(X))}>e^{\alpha t}\right)\\
&\leq e^{-\alpha t}\cdot\mE e^{\alpha(\varphi(X)-\mE\varphi(X))}\\
&\leq e^{-(\alpha t-\psi_+(\alpha))}.
\end{align*}
We use Lemma  \ref{lem:mgf-bound} in the second inequality. Then the upper tail estimate follows by taking the infimum of the right hand side over  $0<\alpha<b$. The lower tail estimate follows from the same argument for $-a<\alpha<0$.
\end{proof}

\section{Log-concavity of moments of $s$-concave functions}
\label{sec:moments}

This section is devoted to the log-concavity of (normalized) moments of $s$-concave functions. This, in conjunction with the results of the previous section, enables us to establish optimal concentration bounds of the information content for convex measures in the following section. 

\begin{defn}\label{defn:s-concave}
For $s\in\mR$, a function $f:\mR^n\to [0, \infty)$ is called $s$-concave if we have
\begin{align}\label{eq:s-concave}
f((1-\lambda)x+\lambda y)\ge \left((1-\lambda) f(x)^s+\lambda f(y)^s\right)^{1/s}
\end{align}
for all $x,y\in\mR^n$ such that $f(x)f(y)>0$ and for all $\lambda\in[0,1]$.
\end{defn}
 
For $s\in\{-\infty, 0, \infty\}$, the RHS of \eqref{eq:s-concave} is defined by taking limits. More precisely, it is equal to $\min\{f(x), f(y)\}$ for $s=-\infty$, $f(x)^{1-\lambda}f(y)^{\lambda}$ for $s=0$, and $\max\{f(x), f(y)\}$ for $s=\infty$. Jensen's inequality implies that the class of $s$-concave functions shrinks with growing $s$. The case $s=0$ corresponds to the class of log-concave functions introduced in the previous section. For $s>0$, the above definition is equivalent to that $f^s$ is concave on its support; while for $s<0$, it is equivalent to that $f^s$ is convex on its support. 

The main result of this section is the following theorem, which unifies and extends previous results of various people. The case $s>0$ was proved by Borell \cite{Bor73a}; the case $s=0$ was proved 
independently by Bolley, Gentil and Guillin  \cite{BGG18} and Fradelizi, Madiman and Wang \cite{FMW16};
it can also be proved by taking a limit of the case $s>0$. The case $s<0$ was proved by Bobkov and Madiman \cite{BM11:it},
except that the range was $p>(n+1)|s|$, and the details of the proof were omitted there because of space considerations. A weaker log-concavity statement was also obtained by Nguyen \cite{Ngu14:1}.

\begin{thm}\label{thm:moments-lebesgue}
Let $s\in\mR$. Let $f:\mR^n\to [0, \infty)$ be an integrable $s$-concave function. The function defined as 
\begin{align*}
\Phi_f(p)=(p+s)\cdots(p+ns)\int_{\mR^n}f(x)^pdx, \ p>0
\end{align*}
and $\Phi_f(0)=n! s^n\vol(\{x\in\mR^n: f(x)>0\})$ is log-concave for $p>\max(0,-ns)$. Moreover, the statement  is sharp in the sense that there exist $s$-concave functions $f$ such that $\Phi_f(p)$ is log-affine.
\end{thm}

By a standard level set argument, Theorem \ref{thm:moments-lebesgue} can be reduced to the log-concavity of moments of $s$-concave functions on the real line. The corresponding result is provided in Proposition \ref{prop:tp-1} below. Recall that the gamma function $\Gamma(x)$ is defined as 
$$
\Gamma(x)=\int_0^{\infty} t^{x-1}e^{-t}dt, \ x>0
$$ 
and the beta function $B(x, y)$ is defined as
$$
B(x ,y) = \int_{0}^{1} t^{x-1} (1- t)^{y-1} dt, \ x,y>0.
$$
We define the following quantity
\begin{align}\label{eq:c-s-p}
C_s(p)=
\left\{  \begin{array}{ll}
B(p,s^{-1}+1)^{-1}  &\text{for}\ s>0\\
 \Gamma(p)^{-1}    &\text{for}\ s=0\\
B(p,-s^{-1}-p)^{-1}   &\text{for}\ s<0.\\  
\end{array} \right.
\end{align}

\begin{prop}\label{prop:tp-1}
Let $s\in\mR$. Let $\varphi: [0,\infty)\rightarrow[0,\infty)$ be an integrable $s$-concave function. Then the function
$$
\Psi_\varphi^s(p)=C_s(p)\int_0^{\infty}t^{p-1}\varphi(t) dt, \ p>0
$$
and $\Psi_\varphi^s(0)=\varphi(0)$ is log-concave for $p\in [0, \infty)$ if $s\geq 0$, and for $p\in[0, -1/s)$ if $s<0$. 
\end{prop}

The case $s>0$ was proved by Borell \cite{Bor73a} except that $\varphi$ is assumed to be decreasing; it was then noticed by some people and available for example in Gu\'edon, Nayar and Tkocz \cite{GNT14} that the result remains true without the monotonicity hypothesis. The case $s<0$ was proved by Fradelizi, Gu\'edon and Pajor \cite{FGP14}, and the case $s=0$ follows by taking a limit. 


\begin{proof}[Proof of Theorem \ref{thm:moments-lebesgue}]


As mentioned before, Theorem \ref{thm:moments-lebesgue} has been proved in the case $s\geq 0$ and the case $s<0$ in a smaller range $p>(n+1)|s|$. We present below a complete proof of the statement in the case $s<0$ for $p>n|s|$. 

\textbf{Log-concavity}. By the change of function $f\to f^{|s|}$, it suffices to prove the statement for $s=-1$. In other words, for any integrable function $f:\R^n\to [0, \infty)$  such that $1/f$ is convex on $K$, the function 
$$\Phi_f(p)= (p-1)\cdots(p-n)\int_{K}f(x)^pdx$$ 
is log-concave on $(n,\infty)$. Denote $g=1/f$, which  is convex on $K$. 
As done by Bobkov and Madiman \cite{BM11:it}, we integrate on level sets and perform a change of variable
$$
\int_{K} f(x)^pdx=\int_{K} g(x)^{-p}dx=\int_0^{\infty}ps^{-p-1}\psi(s)ds=\int_0^{\infty} pt^{p-1}\psi(1/t)dt,
$$
where $\psi(s)=\vol(\{x\in K: g(x)\le s\})$. Using the Brunn-Minkowski theorem, we deduce that $\psi$ is $1/n$-concave, which is equivalent to that $\psi^{1/n}$ is concave. The perspective function of $\psi^{1/n}$ is the bi-variate function $t\psi^{1/n}(s/t)$ for $s, t>0$. We use the property that if a function is convex/concave then its perspective function is also convex/concave. Then it follows that the function $\varphi(t)=t^n\psi(1/t)$ is $1/n$-concave. Thus we get
$$
\int_{\mR^n} f(x)^{p}dx=p\int_0^{\infty} t^{p-n-1}\varphi(t)dt. 
$$
Since $\varphi$ is $1/n$-concave, from Proposition \ref{prop:tp-1} we deduce that 
$$
p\mapsto B(p-n,n+1)^{-1}\int_0^{\infty} t^{p-n-1}\varphi(t)dt
$$ 
is log-concave on $(n,\infty)$. Then we can conclude the proof of the log-concavity of $\Phi_f(p)$ using the identity 
$$
B(p-n,n+1)^{-1}=\frac{p(p-1)\cdots(p-n)}{n!}.
$$

\textbf{Sharpness}. We construct below $s$-concave functions $f$ such that $\Phi_f(p)$ is log-affine. Let $U: \mR^n\to [0, \infty]$ be a positively homogeneous convex function of degree one; that is, $U(tx)=tU(x)$ for all $x\in\mR^n$ and all $t>0$. We define
\begin{align}\label{eq:f-s-u}
f_{s, U}(x)=
\left\{
\begin{array}{ll}
(1-sU(x))_+^{1/s} &\text{for}\ s\neq 0\\
e^{-U(x)} &\text{for}\ s=0.\\
\end{array}
\right.
\end{align}
The sharpness of Theorem \ref{thm:moments-lebesgue} readily follows from the following identity.
\begin{equation}\label{eq:fsU-id}
\int_{\mR^n} f_{s, U}(x)^pdx=\frac{C_Un!}{(p+s)\cdots(p+ns)},
\end{equation}
where $C_U=\vol(\{x\in \mR^n: U(x)\leq 1\})$. We only show identity \eqref{eq:fsU-id} for $s>0$, and the case $s\leq 0$ can be verified in a similar manner.
We have
\begin{align}\label{eq:homogeneity}
\int_{\mR^n} f_{s, U}(x)^pdx &= p\int_0^1t^{p-1}\vol(\{x\in\mR^n: (1-sU(x))_+^{1/s}>t\})dt \notag\\
&= p\int_0^1t^{p-1}\vol\left(\bigg\{x\in\mR^n: U(x)<\frac{1-t^s}{s} \bigg\}\right) dt \notag\\
&= C_Up\int_0^1t^{p-1}\bigg(\frac{1-t^s}{s}\bigg)^ndt\\
&= C_Ups^{-n-1}B\bigg(\frac{p}{s}, n+1\bigg). \notag
\end{align}
In equation \eqref{eq:homogeneity}, we use the homogeneity of $U$ and properties of Lebesgue measure. Then identity \eqref{eq:fsU-id} follows from the fact that
$$
B\bigg(\frac{p}{s}, n+1\bigg)=\frac{n!}{p/s\left(p/s+1\right)\cdots\left(p/s+n\right)}=\frac{s^{n+1}n!}{p(p+s)\cdots(p+ns)}.
$$
\end{proof}


\begin{remark}\label{rmk:extr}
We can rewrite Theorem \ref{thm:moments-lebesgue} in the following equivalent form.  Let $s\in\mR$. Let $f:\mR^n\to [0, \infty)$ be an  integrable $s$-concave function. Then the function
$$
p\mapsto\frac{\int_{\mR^n} f(x)^pdx}{\int_{\mR^n} f_{s, U}(x)^pdx}
$$
is log-concave for $p>\max(0,-ns)$. Here, the function $f_{s, U}$ is defined as per \eqref{eq:f-s-u}. Note that any norm provides an example of a positively homogeneous (of degree one) convex function $U: \mR^n\to [0, \infty]$.
Particularly, simple examples would be either the Euclidean norm $U(x)=|x|$ or the $\ell_1$-norm $U(x)=\sum_{i=1}^n |x_i|$.
Furthermore, the same result holds if one replaces $f_{s,U}$ by $f_{s,U}{\bf 1}_{\mR_+^n}$.
This form is especially convenient when $U$ is the $\ell_1$-norm, so that the function of interest just 
depends on the sum of coordinates. Specifically, this choice would yield the extremizers
\begin{equation}\label{eq:extr}
 f_{s}(x)=\frac{1}{Z(s)}\bigg(1-s\sum_{i=1}^n x_i\bigg)_+^{1/s} {\bf 1}_{\mR_+^n}(x)
\end{equation} 
for $s\neq 0$, where $Z(s)=\prod_{i=1}^n(1+is)^{-1}$; and $f_{0}(x)=e^{-\sum_{i=1}^n x_i} {\bf 1}_{\mR_+^n}(x)$.
\end{remark}

\section{Optimal concentration bounds for information content}
\label{sec:opt}

We are now ready to establish sharp exponential deviation estimates of the information content for convex measures. 
The study of convex measures was initiated by Borell in the seminal papers \cite{Bor74, Bor75a}. 
Given two subsets $A, B\subseteq\mR^n$ and $0<\lambda<1$, we define the Minkowski sum
$$
(1-\lambda)A+\lambda B=\{(1-\lambda) x+\lambda y: x\in A, y\in B\}.
$$
\begin{defn}
Let $\kappa\in \mR\cup\{-\infty, \infty\}$. A finite Borel measure $\mu$ on $\mR^n$ is called $\kappa$-concave if we have
\begin{align}\label{eq:kappa-concave}
\mu((1-\lambda)A+\lambda B)\geq((1-\lambda)\mu(A)^\kappa+\lambda\mu(B)^\kappa)^{1/\kappa}
\end{align}
for all Borel sets $A, B\subseteq \mR^n$ such that $\mu(A)\mu(B)>0$ and for all $\lambda\in[0, 1]$.
\end{defn}

For $\kappa\in\{-\infty, 0, \infty\}$, this definition is interpreted by taking limits. More precisely,  the RHS of \eqref{eq:kappa-concave} is equal to 
$\min(\mu(A), \mu(B))$ for $\kappa=-\infty$; $\mu(A)^{1-\lambda}\mu(B)^{\lambda}$ for $\kappa=0$, and $\max(\mu(A), \mu(B))$ for $\kappa=\infty$. 
Jensen's inequality implies that the class of $\kappa$-concave measures shrinks with growing $\kappa$. For $\kappa=-\infty$, we obtain the largest class, 
whose members are called {\it convex} or {\it hyperbolic measures}. The case $\kappa=0$ corresponds to the class of {\it log-concave measures},
which includes important measures such as Gaussian and exponential measures. The measure $\mu$ is absolutely continuous with respect to 
the Lebesgue measure restricted to an affine subspace $H$ of dimension $0\le d\le n$ such that $\kappa\le 1/d$, 
and its density is $s$-concave with $s=\frac{\kappa}{1-d\kappa}$ (as per Definition \ref{defn:s-concave}); note that therefore  $s>-1/d$. 
In particular, if $\kappa=1/d$ then, up to a scaling factor, $\mu$ is the Lebesgue measure supported on a convex set of $H$.


A random vector $X$ in $\mR^n$ is called $\kappa$-concave if the distribution of $X$ is $\kappa$-concave. 
Suppose that $X$ is a $\kappa$-concave random vector in $\mR^n$ with density $f$. (From the previous 
paragraph, it is necessary that $\kappa\leq 1/n$). Recall that the information content of $X$ is defined as $\widetilde{h}(X)=-\log f(X)$. 
The Laplace transform of $\widetilde{h}(X)$ is
$$
\mE f^{-\alpha}(X)=\int_{\mR^n} f(x)^{1-\alpha}dx.
$$
The integral is finite as long as $\alpha<1+\min\{s, ns\}$. The following statement readily follows from Theorem \ref{thm:moments-lebesgue} with $p$ replaced by $1-\alpha$.

\begin{prop} \label{prop:log-concave}
Let $s\in(-1/n, \infty]$. Let $X$ be a random vector in $\mR^n$ with density $f$ being $s$-concave. Then the function 
\begin{align}\label{eq:function}
\alpha\mapsto\prod_{i=1}^n(1-\alpha+is)\mE f^{-\alpha}(X)
\end{align}
is log-concave for $\alpha<1+\min\{0, ns\}$.
\end{prop}

\begin{remark}\label{rmk:gap}
In the case $s>0$, the statement holds in an interval smaller than the range where the Laplace transform $\mE f^{-\alpha}(X)$ is finite, i.e., $\alpha<1+s$. We suspect that Proposition \ref{prop:log-concave} holds for a larger range $\alpha<1+\min\{s, ns\}$.
\end{remark}

Following Lemma \ref{lem:mgf-bound}, we set 
\begin{align}\label{eq:c(a)}
c(\alpha)=-\sum_{i=1}^n\log(1-\alpha+is).
\end{align}

\begin{coro}\label{cor:v-bound}
Under assumptions and notations of Proposition \ref{prop:log-concave}, we have   
\begin{align} \label{eq:varentropy-bound}
\Var(\widetilde{h}(X))\leq\sum_{i=1}^n(1+is)^{-2}.
\end{align}
\end{coro}

\begin{proof}
By Lemma \ref{lem:exp-family}, we know that $\Var(\widetilde{h}(X_\alpha))=L''(\alpha)$, where the tilted random vector $X_{\alpha}$ has density proportional to $f^{1-\alpha}$ and $L(\alpha)=\log\mE f^{-\alpha}(X)$ is the logarithmic Laplace transform. By Proposition \ref{prop:log-concave}, we know that $L''(\alpha)\leq c''(\alpha)$, where $c(\alpha)$ is defined in \eqref{eq:c(a)}. Then the variance bound \eqref{eq:varentropy-bound} follows by differentiating $c(\alpha)$ twice and setting $\alpha=0$. 
\end{proof}

\begin{remark}
Equality of \eqref{eq:varentropy-bound} holds for a large class of densities of the form $(1-sU(x))_+^{1/s}$, where $U$ is a positively homogeneous convex function of degree 1, i.e., $U(tx)=tU(x)$ for all $x\in\mR^n$ and all $t>0$. In this case, the normalized Laplace transform in \eqref{eq:function} is log-affine, i.e., $L''(\alpha)= c''(\alpha)$. Hence, we have equality in the above variance bound.  Particularly, this class includes Pareto distributions given by \eqref{eq:extr}. 
\end{remark}

\begin{remark}
The limiting case $s=0$ of Corollary~\ref{cor:v-bound} recovers the sharp variance bound $\Var(\widetilde{h}(X))\leq n$ for
log-concave random vectors in $\mR^n$. A bound of this form $\Var(\widetilde{h}(X))\leq Cn$
for an absolute constant $C$ was first obtained by Bobkov and Madiman \cite{BM11:aop} as a consequence
of their concentration result of $\widetilde{h}(X)$. Nguyen \cite{Ngu13:phd} and Wang \cite{Wan14:phd} independently determined that the sharp constant $C=1$. 
Other proofs of the sharp constant were independently given by Bolley, Gentil and Guillin \cite{BGG18} and Fradelizi, Madiman and Wang \cite{FMW16}.
The Fisher information $J(X)$ of a random vector $X$ in $\mR^n$ with density $f$ is defined as
 $$
J(X)=\int_{\mR^n}\frac{|\nabla f|^2}{f}dx.
$$
For isotropic log-concave random vectors $X$ (i.e., $\mE X=0$ and the covariance matrix $\Sigma=\mE [(X-\mE X)\otimes(X-\mE X)]$ is the identity matrix), a well known fact states that $J(X)\geq n$ and equality holds when $X$ is a standard Gaussian random vector. This, together with the sharp variance bound, yields that $\Var(\widetilde{h}(X))\leq J(X)$, which was observed by Nguyen \cite{Ngu14:1}. This relates the varentropy to the Fisher information, and its form may be compared to the logarithmic Sobolev inequality,
which relates the entropy to the Fisher information.
\end{remark}

The following result provides a control of Laplace transform. It readily follows from Lemma \ref{lem:mgf-bound} and Proposition \ref{prop:log-concave}. One can check that equality particularly holds for Pareto distributions defined in \eqref{eq:extr}. 

\begin{thm}\label{thm:main}
Let $s\in(-1/n, \infty)$. Let $X$ be a random vector in $\mR^n$ with density $f$ being $s$-concave. For $\alpha<1+\min\{0, ns\}$, we have
$$
\mE e^{\alpha(\widetilde{h}(X)-h(X))}\leq e^{\psi(\alpha)},
$$
where
\begin{align}\label{eq:psi(a)}
\psi(\alpha)=-\alpha\sum_{i=1}^n(1+is)^{-1}-\sum_{i=1}^n\log\frac{1-\alpha+is}{1+is}.
\end{align}
\end{thm}

\begin{remark} \label{rmk:gap-1}
Similar to the issue mentioned in Remark \ref{rmk:gap}, we can only control the Laplace transform in an interval smaller than the range where it is finite. We also suspect that Theorem \ref{thm:main} holds for a larger range $\alpha<1+\min\{s, ns\}$.
\end{remark}

Then we can apply the Cram\'er-Chernoff method to establish the following sharp concentration bounds on the information content for convex measures.

\begin{cor}\label{cor:devia-est}
Under assumptions and notations of Theorem \ref{thm:main}, we have for any $t>0$, 
$$
\mP\left(\widetilde{h}(X)-h(X)>t\right)\leq e^{-\psi_+^*(t)}
$$
$$
\mP\left(\widetilde{h}(X)-h(X)<-t\right)\leq e^{-\psi_-^*(-t)},
$$
where $\psi_+^*$ and $\psi_-^*$ are Legendre transforms of $\psi_+$ and $\psi_-$, respectively.
\end{cor}

\section{Usable bounds for concentration of information content}
\label{sec:usable}

Next, we put optimal bounds in the previous section into a more usable form. In the log-concave case, it is nearly straightforward to write Corollary \ref{cor:devia-est} in a more transparent way since explicit formulas of $\psi_+^*$ and $\psi_-^*$ are known (see \cite{FMW16}). In the general $s$-concave case, Legendre transforms $\psi_+^*$ and $\psi_-^*$ are implicit and it is nontrivial to derive from Corollary \ref{cor:devia-est} explicit bounds on tail probabilities of the information content. 

\begin{cor}\label{cor:bounds-s>0}
Let $s\in(-1/n, 0]$. Let $X$ be a random vector in $\mR^n$ with density $f$ being $s$-concave. For any $t>0$, we have
\begin{equation}
\mP\left(\widetilde{h}(X)-h(X)>nt\right)  \leq \exp\left(-\frac{n(1+ns)^2}{15}\min\{t, (1+ns)t^2\}\right) \label{eq:tail+}
\end{equation}
\begin{equation}
\mP\left(\widetilde{h}(X)-h(X)<-nt\right)\leq \exp\left(-\frac{n(1+ns)^3}{2}t^2\right). \label{eq:tail-}
\end{equation}
\end{cor}
\begin{proof}
The tail probability estimates \eqref{eq:tail+} and \eqref{eq:tail-} follow from lower bounds of Legendre transforms $\psi_+^*$ and $\psi_-^*$, respectively. We proceed the proof in two cases.

 \textbf{Upper tail}.   
Recall that the function $\psi$ is defined as per \eqref{eq:psi(a)}. Given $u>0$, as a function of $\alpha\in[0, 1+ns)$, one can verify that  
\begin{align}
(\alpha u-\psi(\alpha))'&=u+\sum_{i=1}^n(1+is)^{-1}-\sum_{i=1}^n(1-\alpha+is)^{-1} \label{eq:first-deri}\\
(\alpha u-\psi(\alpha))''&=-\sum_{i=1}^n(1-\alpha+is)^{-2}<0 \label{eq:second-deri}.
\end{align}
It is easy to see that $(\alpha u-\psi(\alpha))'(0)=u>0$ and $\lim_{\alpha\to1+ns}(\alpha u-\psi(\alpha))'=-\infty$. This, together with \eqref{eq:second-deri}, implies that there is an unique $\alpha^*\in(0, 1+ns)$ such that $(\alpha u-\psi(\alpha))'=0$. Moreover, $\alpha u-\psi(\alpha)$ is increasing on $[0, \alpha^*]$ and decreasing on $(\alpha^*, 1+ns)$.
Hence, we have 
\begin{align}\label{eq:sup+}
\psi_+^*(u) = \alpha^* u-\psi(\alpha^*).
\end{align}
We give below an estimate of $\alpha^*$, which will yield a lower bound on  $\psi_+^*(u)$. One can check that
$$
(\alpha u-\psi(\alpha))'''=-2\sum_{i=1}^n(1-\alpha+is)^{-3}<0.
$$
Hence, the function $(\alpha u-\psi(\alpha))'$ is also concave. This concavity implies that
$$
(\alpha u-\psi(\alpha))'(0)+\alpha^*(\alpha u-\psi(\alpha))''(\alpha^*)<(\alpha u-\psi(\alpha))'(\alpha^*).
$$
Combining this inequality with \eqref{eq:first-deri} and \eqref{eq:second-deri}, we have
\begin{equation}\label{eq:lower-alpha-star}
\alpha^* >\frac{u}{\sum_{i=1}^n(1-\alpha^*+is)^{-2}}> \frac{u(1-\alpha^*+ns)}{\sum_{i=1}^n(1-\alpha^*+is)^{-1}}= \frac{u\left(1-\alpha^*+ns\right)}{u+\sum_{i=1}^n(1+is)^{-1}}.
\end{equation}
The identity in \eqref{eq:lower-alpha-star} follows from the fact that $\alpha^*$ is the zero of $(\alpha u-\psi(\alpha))'$, which was given in equation \eqref{eq:first-deri}. We make the following change of variable 
\begin{align}\label{eq:change variable}
u=w\sum_{i=1}^n(1+is)^{-1}.
\end{align} 
Then the lower bound \eqref{eq:lower-alpha-star} can be rewritten as 
\begin{align}\label{eq:lower-alpha-star-1}
\alpha^*>\frac{w}{1+2w}\left(1+ns\right).
\end{align}
Recall that $\alpha u-\psi(\alpha)$ is increasing on $[0, \alpha^*]$. Then we apply \eqref{eq:sup+}, \eqref{eq:psi(a)} and \eqref{eq:lower-alpha-star-1} to obtain
$$
\psi_+^*(u)>\frac{w(1+w)}{1+2w}\sum_{i=1}^n\frac{1+ns}{1+is}+\sum_{i=1}^n\log\left(1-\frac{w}{1+2w}\cdot\frac{1+ns}{1+is}\right).
$$
Owing to the inequality $\log(1-x)\geq-x-4x^2/5$ for $0\leq x\leq1/2$, we have
\begin{align}
\psi_+^*(u) &> \frac{w(1+w)}{1+2w}\sum_{i=1}^n\frac{1+ns}{1+is}-\frac{w}{1+2w}\sum_{i=1}^n\frac{1+ns}{1+is}-\frac{4w^2}{5(1+2w)^2}\sum_{i=1}^n\left(\frac{1+ns}{1+is}\right)^2 \notag\\
&>\frac{w^2}{1+2w}\sum_{i=1}^n\frac{1+ns}{1+is}-\frac{4w^2}{5(1+2w)^2}\sum_{i=1}^n\frac{1+ns}{1+is} \label{eq:second ineq}\\
&>\frac{w^2}{5(1+2w)}\sum_{i=1}^n\frac{1+ns}{1+is} \notag\\
&>\frac{n(1+ns)}{15}\min\{w, w^2\}. \label{eq:last ineq}
\end{align}
Since $s\in(-1/n, 0]$, we have $1+ns<1+is$ for $1\leq i<n$, which yields $\big(\frac{1+ns}{1+is}\big)^2<\frac{1+ns}{1+is}$. This was used in inequality \eqref{eq:second ineq}. The last inequality \eqref{eq:last ineq} follows from $0<1+is<1$ and the simple observation that $\frac{w^2}{1+2w}>\frac{1}{3}\min\{w, w^2\}$. The latter fact can be verified by considering two cases $0<w\leq 1$ and $w>1$.
By Corollary \ref{cor:devia-est}, for any $u>0$ (or, equivalently, $w>0$),
\begin{align}\label{eq:tail+u}
\mP\left(\widetilde{h}(X)-h(X)>u\right)\leq \exp\left(-\frac{n(1+ns)}{15}\min\{w, w^2\}\right).
\end{align}
Identity \eqref{eq:change variable}, together with $1+ns<1+is$ for $1\leq i<n$, implies that $u<\frac{nw}{1+ns}$. Then, by \eqref{eq:tail+u}, we have that for any $w>0$,
$$
\mP\left(\widetilde{h}(X)-h(X)>\frac{nw}{1+ns}\right)\leq \exp\left(-\frac{n(1+ns)}{15}\min\{w, w^2\}\right).
$$
Then the upper tail estimate \eqref{eq:tail+} follows from the change of variable $t=\frac{w}{1+ns}$.

\textbf{Lower tail}. The lower tail estimate \eqref{eq:tail-} can be proved in a similar manner. Recall that the function $\psi$ is defined as per \eqref{eq:psi(a)}. As a function of $\alpha<0$, it is easy to see from \eqref{eq:first-deri} that $(-u\alpha-\psi(\alpha))'<0$ for $u\geq\sum_{i=1}^n(1+is)^{-1}$. In this case, we have $\psi_-^*(-u)=\infty$. This, together with Corollary \ref{cor:devia-est}, yields that $\widetilde{h}(X)-h(X)\geq -\sum_{i=1}^n(1+is)^{-1}$ with probability one. Hence, the lower tail estimate \eqref{eq:tail-} trivially holds for $t>n^{-1}\sum_{i=1}^n(1+is)^{-1}$. Next, we will assume that $0<u<\sum_{i=1}^n(1+is)^{-1}$. Similar to the upper tail case, one can check that there is an unique solution of $(-u\alpha-\psi(\alpha))'=0$, which we denote by $\alpha^*$ by adopting the abuse of notation. Moreover, the function $-u\alpha-\psi(\alpha)$ is increasing on $(-\infty, \alpha^*)$ and decreasing on $[\alpha^*, 0]$. Hence, we have
\begin{align}\label{eq:sup-}
\psi_-^*(-u) = -\alpha^* u-\psi(\alpha^*).
\end{align}
Identity \eqref{eq:second-deri} still holds with $u$ replaced by $-u$. Hence, $(-u\alpha-\psi(\alpha))'$ is concave. This concavity implies that
$$
(-u\alpha-\psi(\alpha))'(\alpha^*)-\alpha^*(-u\alpha-\psi(\alpha))''(0)<(-u\alpha-\psi(\alpha))'(0).
$$
Combining this with \eqref{eq:first-deri} and \eqref{eq:second-deri} with $u$ replaced by $-u$, we have
\begin{equation}\label{eq:upper-alpha-star}
\alpha^*<\frac{-u}{\sum_{i=1}^n(1+is)^{-2}}<\frac{-u(1+ns)}{\sum_{i=1}^n(1+is)^{-1}}=-w(1+ns).
\end{equation}
The second inequality uses $1+ns<1+is$ for $1\leq i<n$ and the last identity follows from the change of variable \eqref{eq:change variable}. As mentioned before, $-\alpha u-\psi(\alpha)$ is decreasing on $[\alpha^*, 0]$. Plug the upper bound of $\alpha^*$ given in \eqref{eq:upper-alpha-star} into \eqref{eq:sup-} and invoke the assumption $0<u<\sum_{i=1}^n(1+is)^{-1}$ to obtain
\begin{align}\label{eq:psi-minus-lower}
\psi_-^*(-u)> -w(1-w)\sum_{i=1}^n\frac{1+ns}{1+is}+\sum_{i=1}^n\log\left(1+w\cdot\frac{1+ns}{1+is}\right).
\end{align}
Owing to the inequality $\log(1+x)>x-x^2/2$ for $x\geq0$, we have
\begin{align*}
\psi_-^*(-u) &> -w(1-w)\sum_{i=1}^n\frac{1+ns}{1+is}+w\sum_{i=1}^n\frac{1+ns}{1+is}-\frac{w^2}{2}\sum_{i=1}^n\left(\frac{1+ns}{1+is}\right)^2\\
&>\frac{w^2}{2}\sum_{i=1}^n\frac{1+ns}{1+is}>\frac{n(1+ns)}{2}w^2.
\end{align*}
Similar to the upper tail case, we applied $\big(\frac{1+ns}{1+is}\big)^2<\frac{1+ns}{1+is}$ in the second inequality and $0<1+is<1$ in the last inequality. By Corollary \ref{cor:devia-est}, for any $u>0$ (or, equivalently, any $w>0$), we have
\begin{align}\label{eq:tail+u-1}
\mP\left(\widetilde{h}(X)-h(X)<-u\right)\leq \exp\left(-\frac{n(1+ns)}{2}w^2\right).
\end{align}
Identity \eqref{eq:change variable}, together with $1+ns<1+is$ for $1\leq i<n$, implies that $u<\frac{nw}{1+ns}$. This together with \eqref{eq:tail+u-1} implies that for any $w>0$,
$$
\mP\left(\widetilde{h}(X)-h(X)<-\frac{nw}{1+ns}\right)\leq \exp\left(-\frac{n(1+ns)}{2}w^2\right).
$$
Then the lower tail estimate \eqref{eq:tail-} readily follows the change of variable $t=\frac{w}{1+ns}$.
\end{proof}

The following result provides bounds of tail probabilities of the information content for $s$-concave densities in the range $s>0$. This, together with Corollary \ref{cor:bounds-s>0}, gives a complete picture of the concentration phenomenon of the information content for different ranges of convex measures. We only give a proof sketch of the statement, since the proof is the same as that of Corollary \ref{cor:bounds-s>0} with a minor change of notations.

\begin{cor}
Let $s\in(0, \infty)$. Let $X$ be a random vector in $\mR^n$ with density $f$ being $s$-concave. For any $t>0$, we have
\begin{equation}
\mP\left(\widetilde{h}(X)-h(X)>nt\right) \leq (en)^{1/s}\cdot e^{-nt} \label{eq:upper-tail+s}
\end{equation}
\begin{equation}
\mP\left(\widetilde{h}(X)-h(X)<-nt\right) \leq \exp\left(-\frac{n(1+s)^3}{2}t^2\right) \label{eq:lower-tail+s}.
\end{equation}
\end{cor}

\begin{proof}
The upper tail estimate \eqref{eq:upper-tail+s} readily follows from $\psi_+^*(u)\geq u-\psi(1)$ for any $u>0$ and the simple observation
$$
\psi(1)=\sum_{i=1}^n\log \left(1+\frac{1}{is}\right)-\sum_{i=1}^n\frac{1}{1+is}<\sum_{i=1}^n\log \left(1+\frac{1}{is}\right)<\sum_{i=1}^n\frac{1}{is}<\frac{1+\log n}{s}.
$$ 
The lower tail estimate \eqref{eq:lower-tail+s} trivially holds for $t>n^{-1}\sum_{i=1}^n(1+is)^{-1}$. This follows from the fact that $\widetilde{h}(X)-h(X)\geq -\sum_{i=1}^n(1+is)^{-1}$ still holds with probability one in the case $s>0$. In the complementary case, the argument of Corollary \ref{cor:bounds-s>0} yields an unique $\alpha^*$ such that \eqref{eq:sup-} holds. Instead of \eqref{eq:upper-alpha-star}, we have the upper bound 
$$
\alpha^*\leq-w\left(1+s\right).
$$
Similar to \eqref{eq:psi-minus-lower}, we have
$$
\psi_-^*(-u)> -w(1-w)\sum_{i=1}^n\frac{1+s}{1+is}+\sum_{i=1}^n\log\left(1+w\cdot\frac{1+s}{1+is}\right).
$$
Following the proof of Corollary \ref{cor:bounds-s>0}, we have
$$
\psi_-^*(-u)>\frac{n(1+s)}{2}w^2.
$$
Then the lower tail estimate \eqref{eq:lower-tail+s} follows from Corollary \ref{cor:devia-est} and the simple fact that for $s>0$
$$
\sum_{i=1}^n(1+is)^{-1}<\frac{n}{1+s}.
$$
\end{proof}

\begin{remark}
When $0<nt<\sum_{i=1}^n(is(1+is))^{-1}$, we can follow the proof of Corollary \ref{cor:bounds-s>0} (with a minor change of notations) to obtain
\begin{align}\label{eq:upper-tail+s-1}
\mP\left(\widetilde{h}(X)-h(X)>nt\right)\leq \exp\left(-\frac{n(1+s)^2}{15}\min\{t, (1+s)t^2\}\right).
\end{align}
This upper tail estimate will hold for any $t>0$ if Theorem \ref{thm:main} holds for $\alpha<1+s$ and $s>0$.
\end{remark}

\section{Monotonicity of moments of $s$-concave functions}
\label{sec:comp}

Let $s\in\mR$. Let $\varphi:[0,\infty)\to[0,\infty)$ be an integrable $s$-concave function that is right continuous at $0$. The Mellin transform  $\M_{\varphi}(p)$ is defined by 
\begin{equation}\label{eq:mellin}
\M_{\varphi}(p)=\int_0^{\infty} t^{p-1} \varphi(t) dt. 
\end{equation}
It is not difficult to see that  $\M_{\varphi}(p)$ is well defined for $p>0$ if $s\ge0$, and for $0<p<-1/s$ if $s<0$. 
For $s<0$, this follows from the fact that $\varphi(t)=O(t^{1/s})$ as $t\to\infty$. 
In fact, the Mellin transform $\M_\varphi(p)$ is analytic on the half plane $\{p\in\C: \re(p)>0\}$ for $s\ge0$, and on the strip 
$\{p\in\C: 0<\re(p)<-1/s\}$ for $s<0$. 

Then the function $\Psi_\varphi^s(p)$ introduced in Proposition \ref{prop:tp-1} can be rewritten as 
\begin{align}\label{eq:psi-func}
\Psi_\varphi^s(p)=C_s(p)\M_{\varphi}(p), \ p>0
\end{align}
and $\Psi_\varphi^s(0)=\varphi(0)$. 
One can check that $\M_{\varphi_s}(p)=C_s(p)^{-1}$, where $C_s(p)$ is defined in \eqref{eq:c-s-p} and the function $\varphi_s$ is defined as
\begin{align}\label{eq:varphi-s}
\varphi_s(t)=\left\{  \begin{array}{ll}
(1-t)^{1/s}{\bf{1}}_{[0,1]}  &\text{for}\ s>0\\
 e^{-t}\bf{1}_{\mR_+}  &\text{for}\ s=0\\
(1+t)^{1/s}\bf{1}_{\mR_+}  &\text{for}\ s<0.\\  
    \end{array} \right.
\end{align}
Hence, for $p>0$, we can rewrite $\Psi_\varphi^s(p)$ as the normalized Mellin transform
\begin{align}\label{eq:psi-func-1}
\Psi_\varphi^s(p)=\frac{\M_{\varphi}(p)}{\M_{\varphi_s}(p)}.
\end{align}

The log-concavity of $\Psi^s_\varphi(p)$ proved in Proposition \ref{prop:tp-1} implies that for any $r\ge0$, the function defined as
 $$
 p\mapsto
 \left\{  \begin{array}{ll}
 \left(\frac{\Psi^s_\varphi(p+r)}{\Psi^s_\varphi(r)}\right)^{1/p} & \text{for}~p\neq0\\
 \exp((\ln \Psi^s_\varphi)'(r)) & \text{for}~p=0
 \end{array}\right. 
 $$
is non-increasing. Recall that $\Psi_\varphi^s(0)=\varphi(0)$. Set $r=0$. If $\varphi(0)>0$, this particularly implies that the function defined by
\begin{equation}\label{eq:q-func}
Q^s_\varphi(p)
=\left(\frac{\Psi^s_\varphi(p)}{\Psi^s_\varphi(0)}\right)^{1/p}, \ p>0
\end{equation}
and $Q^s_\varphi(0)=\exp((\ln \Psi^s_\varphi)'(0))$ is non-increasing on $[0,\infty)$ for $s\ge0$, and on $[0,-1/s)$ for $s<0$. 


Suppose that $\varphi$ is not merely continuous but also right differentiable at $0$. Similar to the gamma function, the Mellin transform $\M_{\varphi}(p)$ can be extended to a meromorphic function in the domain $D_s=\{p\in\C: \re(p)>-1\}$ for $s\ge0$ and 
$D_s=\{p\in\C: -1<\re(p)<-1/s\}$ for $s<0$ with a simple pole at $0$ by the formula 
$$
\M_\varphi(p)=\frac{1}{p}\int_0^{\infty}t^p(-\varphi'(t))dt=\int_0^1t^{p-1} (\varphi(t)-\varphi(0))dt+\frac{\varphi(0)}{p}+\int_1^{\infty}t^{p-1} \varphi(t)dt.
$$
Using integration by parts, for $\re(p)>0$, we can recover the  definition of Mellin transform in \eqref{eq:mellin}, while for $-1<\re(p)<0$ we have a simpler form
\begin{equation}\label{eq:mellin-1}
\M_\varphi(p)=\int_0^{\infty}t^{p-1} (\varphi(t)-\varphi(0))dt.
\end{equation}
Particularly, the Mellin transform $\M_{\varphi_s}$ of the function $\varphi_s$ defined as per \eqref{eq:varphi-s} is meromorphic in the domain $D_s$ with one simple pole at $0$. 
Moreover,  $\varphi_s$  doesn't vanish on $D_s$ and $\varphi_s(0)=1$.
Hence, the function $\Psi^s_\varphi$ defined as per \eqref{eq:psi-func-1} has an analytic extension to the domain $D_s$. Correspondingly, we can analytically extend the function $Q^s_\varphi(p)$ defined as per \eqref{eq:q-func} to the domain $D_s^\R=\{p\in \R\cap D_s; \Psi^s_{\varphi}(p)>0\}$.   

In the case $s=0$, Koldobsky, Pajor and Yaskin \cite{KPY08} 
showed that $Q^0_\varphi$ is non-increasing on  $(-1,\infty)$ under the assumption that $\varphi$ is non-increasing. 
In the following, we extend this monotonicity property to the whole range of $s$-concave functions $\varphi$ 
and we also drop the monotonicity hypothesis on $\varphi$.

\begin{thm}\label{th:monot-phi}
Let $s\in\R$. Let $\varphi:[0,\infty)\to[0,\infty)$ be an integrable $s$-concave function such that $\varphi(0)>0$ and that $\varphi$ is right differentiable at $0$. Set $p_0=\inf\{p>-1: \Psi^s_\varphi(p)>0\}$. Then we have 
\begin{enumerate}
\item $p_0\in[-1,0)$ and if $\varphi$ is non-increasing then $p_0=-1$.
\item $\Psi^s_\varphi(p)>0$ for every $p\in(p_0,0]$, thus $Q^s_\varphi(p)$ is well defined and analytic on $(p_0,\infty)$ for $s\ge0$ and on $(p_0,-1/s)$ for $s<0$.
\item $Q^s_\varphi(p)$ is non-increasing on $(p_0,\infty)$ for $s\ge0$ and on $(p_0,-1/s)$ for $s<0$. In particular, for $s<0$, the function
$$
Q^s_\varphi(p)= \left( C_s(p)\int_0^{\infty} t^{p-1} \frac{\varphi(t)-\varphi(0)}{\varphi(0)} dt \right)^{1/p}
$$
is non-increasing on $(p_0,0)$, where the constant $C_s(p)$ is defined in \eqref{eq:c-s-p}.
\end{enumerate}
\end{thm}

\begin{proof}  
To see the first statement, we recall that $\Psi^s_\varphi(0)=\varphi(0)$. This, together with the assumption that $\varphi(0)>0$, readily yields that $p_0\in[-1,0)$. 
For $-1<p<0$, the Mellin transform $\M_\varphi(p)$ is defined as per \eqref{eq:mellin-1}. If we further assume that $\varphi$ is non-increasing, then it is easy to see that 
$\M_\varphi(p)< 0$. Particularly, we have $\M_{\varphi_s}(p)<0$ since $\varphi_s$ defined in \eqref{eq:varphi-s} is a decreasing function. Hence, we have $\Psi^s_\varphi(p)=\M_\varphi(p)/\M_{\varphi_s}(p)>0$ for $-1<p<0$. This yields that $p_0=-1$ under the monotonicity assumption on $\varphi$.

Next, we prove the second statement. By homogeneity, we may assume that $\varphi(0)=1$. Suppose that $\Psi^s_\varphi(q)>0$ for some $-1<q<0$. Then we have $Q_\varphi^s(q)=(\Psi^s_\varphi(q))^{1/q}>0$. (Here, we use the fact that $\Psi^s_\varphi(0)=\varphi(0)$ and the assumption that $\varphi(0)=1$). We define $\psi(t)=\varphi_s(t/Q_\varphi^s(q))$. One can check that $\M_\psi(p)=(Q_\varphi^s(q))^p\M_{\varphi_s}(p)$ for every $-1<p<0$. In particular, we have
$\M_\psi(q)=(Q_\varphi^s(q))^q\M_{\varphi_s}(q)=\M_\varphi(q)$. Since
$$
\M_\varphi(q)-\M_\psi(q)=\int_0^{\infty}t^{q-1}(\varphi(t)-\psi(t))dt=0,
$$
one can deduce that $\varphi-\psi$ changes sign at least once on $(0,\infty)$. Notice that $\varphi(0)=\psi(0)$, $\varphi$ is $s$-concave and $\psi$ is $s$-affine. It follows that $\varphi-\psi$ changes sign at most once. Hence, $\varphi-\psi$ changes sign exactly at one point $t_0>0$ and by concavity one has necessarily that $\varphi-\psi\ge0$ on $(0,t_0)$ and $\varphi-\psi\le0$ on $(t_0,\infty)$. 
We define $H(t)=\int_t^{\infty}u^{q-1}(\varphi(u)-\psi(u))du.$
Then we have $H'(t)=-t^{q-1}(\varphi(t)-\psi(t))$. Hence, $H(t)$ is non-increasing on $[0,t_0]$ and non-decreasing on $[t_0,\infty)$. 
Since $H(0)=0$ and $H(\infty)=0$, it follows that $H\le0$ on $[0,\infty)$. Using integration by parts, for $-1<q<r<0$, we have
$$
\M_\varphi(r)-\M_\psi(r)=\int_0^{\infty}t^{r-q}t^{q-1}(\varphi(t)-\psi(t))dt=(r-q)\int_0^{\infty}t^{r-q-1}H(t)dt\le0.
$$
Therefore, we have
$$
\M_\varphi(r)\le\M_\psi(r)=(Q_\varphi^s(q))^r\M_{\varphi_s}(r)<0.
$$
The last inequality follows from the facts that $Q_\varphi^s(q)>0$ and that $\M_{\varphi_s}(r)<0$. The latter fact was mentioned in the proof of the first statement. Then we have
\begin{equation}\label{eq:monoton}
\Psi^s_\varphi(r)=\frac{\M_\varphi(r)}{\M_{\varphi_s}(r)} \ge Q_\varphi^s(q)^r > 0.
\end{equation} 
Hence, $Q^s_\varphi$ is well defined and analytic on $(p_0,\infty)$ for $s\ge0$ and on $(p_0,-1/s)$ for $s<0$. 

For the third statement, we have mentioned in the paragraph after \eqref{eq:psi-func-1} that the monotonicity of $Q_\varphi^s$ on $(0,\infty)$ for $s\ge0$ and on $(0,-1/s)$ for $s<0$ follows from the log-concavity of $\Psi^s_\varphi(p)$ that was proved in Proposition \ref{prop:tp-1}. Hence, it suffices to prove the monotonicity on $(p_0, 0)$. Taking the $r$-th root of \eqref{eq:monoton}, for $p_0<q<r<0$, we have
$$
Q_\varphi^s(r)=(\Psi^s_\varphi(r))^{1/r}\le Q_\varphi^s(q).
$$
This shows that $Q_\varphi^s$ is non-increasing on $(p_0,0)$.
\end{proof}

Using Theorem \ref{th:monot-phi}, we establish the following theorem, which extends
results of Borell \cite{Bor73a} and Fradelizi, Gu\'edon and Pajor \cite{FGP14}. 

\begin{thm}\label{th:monot-f}
Let $s\in\R$ and let $\mu$ be a $s$-concave probability measure on $\R^n$. Let $f:\R^n\to[0,\infty)$ be a concave function on its support. Then the function 
$$
p\mapsto  \left( \frac{C_s(p)}{p}\int_{\R^n} f(x)^pd\mu(x) \right)^{1/p}
$$
is non-increasing on $(-1,\infty)$ if $s\ge0$ and on $(-1, -1/s)$ if $s<0$, where the constant $C_s(p)$ is defined in (\ref{eq:c-s-p}).
\end{thm}

\begin{proof}
For $s\ge0$, the monotonicity on $(0,\infty)$ follows from Borell \cite{Bor73a}.
For $s<0$, the monotonicity on $(0, -1/s)$ is due to Fradelizi, Gu\'edon and Pajor \cite{FGP14}. 
It suffices to prove extensions of both cases on $(-1,0)$. 
Let $-1<p<0$. Define $\varphi(t)=\mu(\{x\in\R^n: f(x)>t)\})$ for $t\ge0$. Since $f$ is concave and $\mu$ is $s$-concave, we deduce that $\varphi$ is $s$-concave. Integrating on the level sets, we have 
\begin{align*}
\int_{\mR^n} f(x)^pd\mu(x)&=-\int_{\mR^n} \int_{f(x)}^{\infty}pt^{p-1}dtd\mu(x)=p\int_0^{\infty}t^{p-1}(\varphi(t)-\varphi(0))dt=p\M_\varphi(p).
\end{align*}
The second identity follows from Fubini's theorem. The last identity follows from the definition of Mellin transform $\M_\varphi(p)$ defined as per \eqref{eq:mellin-1}. Then the desired statement follows from the third statement of Theorem \ref{th:monot-phi} and the fact that $\varphi(0)=1$.
\end{proof}

Let $\mu$ be the uniform probability measure on a convex body $K$ of $\R^n$. By the Brunn-Minkowski theorem, it is known that $\mu$ is $1/n$-concave. Notice that 
$$
C_{1/n}(p)=B(p,n+1)^{-1}=\frac{p(p+1)\cdots(p+n)}{n!}=p{n+p \choose n}.
$$ 
Then we can apply Theorem \ref{th:monot-f} to recover a result of Borell \cite{Bor73b} (see, p. 435, also Theorem 5.1 of \cite{GZ98}), which generalizes the classical theorem of Berwald \cite{Ber47}, see also \cite{MP89}, which was restricted to the range $p>0$.

\begin{cor}\label{cor:monot-f-K}
Let $K$ be a convex body of $\R^n$. Let $f:K\to[0,\infty)$ be a concave function. Then the function 
$$
p\mapsto  \left( {n+p \choose n}\frac{1}{|K|}\int_K f(x)^pdx \right)^{1/p}
$$
is non-increasing on $(-1,\infty)$.
\end{cor}

As mentioned before, the monotonicity of $Q^s_\varphi$ follows from the log-concavity of $\Psi^s_\varphi$. We conclude this section with the following conjecture.

\begin{conj}\label{conj:log-concavity}
Let $s\in\R$. Let $\varphi:[0,\infty)\to[0,\infty)$ be an integrable $s$-concave function such that $\varphi(0)>0$ and  that $\varphi$ is right differentiable at $0$. 
Then the function 
$$
\Psi^s_\varphi(p)= C_s(p)\int_0^{\infty} t^{p-1} (\varphi(t)-\varphi(0)) dt
$$
is log-concave on $[p_0, 0)$, where $p_0=\min\{p>-1: \Psi^s_\varphi(p)>0\}$.
\end{conj}

\begin{remark}
Using a standard level set argument, if Conjecture \ref{conj:log-concavity} were true, then it could be used to show that the 
log-concavity statement in Theorem \ref{thm:moments-lebesgue} holds in a larger range of $p>\max(-s,-ns)$. Correspondingly, Proposition \ref{prop:log-concave} and Theorem \ref{thm:main} will hold for $\alpha<1+\min\{s, ns\}$ and the upper tail estimate \eqref{eq:upper-tail+s-1} will hold for any $t>0$.
We can prove this extension of Theorem \ref{thm:moments-lebesgue} when $f$ is defined on $\mR$, but do not
include the details in this paper.
\end{remark}

\section{Consequences} 
\label{sec:conseq}

In this section, we develop various consequences of the results obtained. First, we give a probabilistic interpretation of Theorem~\ref{thm:moments-lebesgue}, namely  
that all R\'enyi entropies of a $s$-concave density are effectively comparable.
Recall that, for $p\in (0, 1)\cup(1, \infty)$, the R\'enyi entropy of order $p$ of a density $f$ is defined as
\begin{equation}\label{eq:renyi}
h_p(X)=\frac{1}{1-p}\log\int_{\mR^n}f(x)^pdx.
\end{equation}
For $p\in\{0, 1, \infty\}$, it is defined in the natural way by taking limits; that is,
$h_1(f)=h(f)$ is the Shannon-Boltzmann entropy; $h_\infty(f)=-\log \|f\|_{\infty}$, where $\|f\|_{\infty}$ is the essential supremum of $f$; and $h_0(f)=\log|\supp\{f\}|$, where $|\supp\{f\}|$ is the Lebesgue measure of the support of $f$. It readily follows
from Jensen's inequality that $h_p(f) \geq h_q(f)$ whenever $q>p\geq 0$ for any density $f$.
The notable fact is that for $s$-concave densities, this inequality can be reversed up to a precise
dimensional constant.

\begin{cor}\label{cor:inf-h}
Let $s\in(-1/n, \infty)$. Let $X$ be a random vector in $\mR^n$ with density $f$ being $s$-concave. 
For $\max(0,-ns)\leq p< q\leq \infty$, we have
$$
h_p(f)-h_q(f)\leq h_p(f_{s})-h_q(f_{s}) ,
$$
where the family of extremizers $f_s$  are defined in \eqref{eq:extr}.
\end{cor}

\begin{proof}
From Remark \ref{rmk:extr}, we know that if $f:\mR^n\to[0, \infty)$ is an  integrable $s$-concave function, then
$$
p\mapsto \frac{\int_{\mR^n}f(x)^pdx}{\int_{\mR^n}f_{s}(x)^pdx}
$$
is log-concave for $p>\max\{0,-ns\}$. If $f$ is a $s$-concave density, then in the language of R\'enyi entropies, one deduces that
$$
\xi(p)=(1-p) (h_p(f) -h_p(f_{s}))
$$
is concave for $p>\max\{0,-ns\}$.  Moreover, given that both $f$ and $f_{s}$ integrate to one, we should set $\xi(1)=0$; with this choice, $\xi$ is concave on its maximal interval of definition. It is a well known fact (see, e.g., \cite[Exercise 3.1]{BV04:book}) that a function $\xi: (a,b)\to\mR$ is concave if and only if
$$
\frac{\xi(p)-\xi(\tilde{p})}{p-\tilde{p}}
$$ 
is non-increasing in $p$ (for fixed $\tilde{p}$) and non-decreasing in $\tilde{p}$ (for fixed $p$). Taking $\tilde{p}=1$, we see that
$$
h_p(f) -h_p(f_{s})
=\frac{\xi(p)}{1-p}
=-\frac{\xi(p)-\xi(1)}{p-1} 
$$ 
is non-decreasing in $p$, i.e., we have that
$$
h_p(f) -h_p(f_{s})
\leq h_q(f) -h_q(f_{s})
$$
for any $\max\{0,-ns\}< p<q<\infty$. Rearranging gives the desired inequality in this range. Taking limits
extends the range to $q=\infty$ and to $p=\max\{0,-ns\}$.
\end{proof}

\begin{remark}
One can check that for $p\in (0,1)\cup(1,\infty)$,
\begin{equation}\label{eq:Lsp1}
h_p(f_{s})=\frac{p}{1-p}\sum_{i=1}^n\log(1+is)-\frac{1}{1-p}\sum_{i=1}^n\log(p+is).
\end{equation}
We also have the limiting cases
\begin{align}
h(f_{s}) & =\sum_{i=1}^n(1+is)^{-1}-\sum_{i=1}^n\log(1+is) \label{eq:Lsp2}\\
h_\infty(f_{s}) & =-\sum_{i=1}^n\log(1+is) \label{eq:Lsp3}.
\end{align}
Hence, for a $s$-concave density $f$, Corollary~\ref{cor:inf-h} yields, when $p<q$, the numerical bound
\begin{equation}\label{eq:compare}
h_p(f) - h_q(f)
\leq \frac{p-q}{(1-p)(1-q)}\sum_{i=1}^n\log(1+is)-\sum_{i=1}^n\log \frac{(p+is)^\frac{1}{1-p}}{(q+is)^\frac{1}{1-q}}
\end{equation}
with the cases where $p$ or $q$ is 0, 1 or $\infty$ understood by taking limits.
\end{remark}

\begin{remark}\label{rmk:comparison-s=0}
The log-concave case of Corollary~\ref{cor:inf-h} was first observed in \cite{MW19}. 
Also, taking $p=1$ and $q=\infty$ in \eqref{eq:compare}, we have
$$
h(f)-h_\infty(f)\leq \sum_{i=1}^n(1+is)^{-1}.
$$
The special case of this inequality with $-1/(n+1)\leq s\leq 0$ was obtained earlier in \cite{BM11:it}.  
\end{remark}

Second, we present an improvement of \cite[Proposition 5.1]{BM12:jfa} (whose analogue for the special case of log-concave probability 
measures was first observed by Klartag and Milman \cite{KM05} and later refined in \cite[Corollary 4.7]{FMW16}).

\begin{cor}\label{coro:ess-supp}
Let $s\in(-1/n, 0)$. Let $X$ be a random vector in $\mR^n$ with density being $s$-concave. For any $c_0\in(0, 1)$ such that $n\log c_0<-\sum_{i=1}^n(1+is)^{-1}$, there exists $c_1\in(0, 1)$ depending only on $c_0$ and $s$ such that
$$
\mP(f(X)\geq c_0^n\|f\|_\infty)\geq 1-c_1^n.
$$

\end{cor}

\begin{proof}
Set $t=-n\log c_0-\sum_{i=1}^n(1+is)^{-1}$. By Remark \ref{rmk:comparison-s=0}, we have
$$
\mP(f(X)\leq c_0^n\|f\|_\infty)= \mP(\widetilde{h}(X)\geq -\log\|f\|_\infty-n\log c_0)\leq \mP(\widetilde{h}(X)-h(X)\geq t). 
$$
Then, by Corollary \ref{cor:devia-est}, we have
\begin{equation}\label{eq:dev-f-infty}
\mP(f(X)\leq c_0^n\|f\|_\infty)\leq e^{-\psi_+^*(t)}.
\end{equation}
The function $\psi$ is defined as per \eqref{eq:psi(a)}. By \eqref{eq:second-deri}, we know that  $\alpha t-\psi(\alpha)$ is concave for $\alpha<1+ns$. Hence, the maximum is achieved at $\alpha^*$ such that $(\alpha t-\psi(\alpha))'=0$,
i.e., 
\begin{align}\label{eq:alpha*}
\sum_{i=1}^n(1-\alpha^*+is)^{-1}=-n\log c_0.
\end{align}
Using \eqref{eq:psi(a)}, we have
\begin{align*}
\psi_ +^*(t) &= \alpha^* t-\psi(\alpha^*)= -n\alpha^*\log c_0+\sum_{i=1}^n\log\frac{1-\alpha^*+is}{1+is}.
\end{align*}
Plug this into \eqref{eq:dev-f-infty} to obtain
$$
\mP(f(X)\leq c_0^n\|f\|_\infty) \leq c_1^n,
$$
where
$$
c_1=c_0^{\alpha^*}\left(\prod_{i=1}^n\frac{1+is}{1-\alpha^*+is}\right)^{1/n}.
$$
This is equivalent to the desired statement. To see that $c_1<1$, we take the logarithm of $c_1$,
\begin{align}
\log c_1 &= \alpha^*\log c_0+\frac{1}{n}\sum_{i=1}^n\log\frac{1+is}{1-\alpha^*+is} \notag \\
&= -\frac{1}{n}\sum_{i=1}^n\frac{\alpha^*}{1-\alpha^*+is}+\frac{1}{n}\sum_{i=1}^n\log\frac{1+is}{1-\alpha^*+is} \label{eq:plug alpha*} \\
&=-\frac{1}{n}\sum_{i=1}^n\left(\frac{\alpha^*}{1-\alpha^*+is}-\log\left(1+\frac{\alpha^*}{1-\alpha^*+is}\right)\right) \notag\\
&<0. \notag
\end{align}
We used equation \eqref{eq:alpha*} in the second identity \eqref{eq:plug alpha*}. The inequality follows from $\log(1+x)<x$ for $x>0$.
\end{proof}

\begin{remark}
For $s>0$, Theorem \ref{thm:main} holds for $\alpha<1$. In this range, one might not find the solution $\alpha^*$ of equation \eqref{eq:alpha*}. The proof of Corollary \ref{coro:ess-supp} will break down. Assuming the validity of Conjecture \ref{conj:log-concavity}, Theorem \ref{thm:main} will hold in a larger range $\alpha<1+s$. Then we can find such an $\alpha^*$. Following the same argument, Corollary \ref{coro:ess-supp} will also hold for any $s>0$.
\end{remark}

Third, we show that the entropy power of a convex measure is linked to the volume of a typical set or effective support. 

\begin{cor}
Under assumptions and notations of Corollary \ref{coro:ess-supp}, we have
$$
c_0^n|K|\leq e^{h(f)}\leq(1-c_1^n)^{-1}\exp\left(\sum_{i=1}^n(1+is)^{-1})\right)|K|,
$$
where the convex set $K=\{x\in \mR^n: f(x)\geq c_0^n\|f\|_\infty\}$. 
\end{cor}
\begin{proof}
From the definition of $K$ we can see that
$$
c_0^n\|f\|_\infty |K|\leq \int_Kf(x)dx\leq 1.
$$
Corollary \ref{coro:ess-supp} implies that
$$
1-c_1^n\leq \int_Kf(x)dx\leq \|f\|_\infty|K|.
$$
Remark \ref{rmk:comparison-s=0} implies the entropy power estimate
$$
\|f\|_\infty^{-1}\leq e^{h(f)}\leq \exp\left(\sum_{i=1}^n(1+is)^{-1}\right)\|f\|_\infty^{-1}.
$$
Then the desired statement follows by combining the above inequalities.
\end{proof}

The entropy power (usually defined as $e^{h(f)}$ raised to the power $2/n$) is a notion with origins
in information theory. A class of inequalities for the entropy power of convolutions, known as entropy power
inequalities, play an important role in convex geometry, information theory, and probability theory, particularly
in connection with the central limit theorem, and may also be seen as probabilistic parallels to Brunn-Minkowski
and related inequalities for set volumes. For more on these connections, the reader may consult \cite{MMX17:0, FMMZ16, FMMZ18, Mad20};
in particular, the preceding corollary can be used to go back and forth between volumes and entropies in drawing
parallels between geometry and probability in the context of convex measures.

Fourth, we recover a result of Adamczak et al. \cite[Lemma 7.2]{AGLLOPT12},
which in turn generalizes a result of Fradelizi \cite{Fra97}
relating the maximum value of a log-concave density to its value at the mean.

\begin{cor}\label{coro:fra}
Let $s\in(-\frac{1}{n+1}, \infty)$. Let $X$ be a random vector in $\R^n$ with density $f$ being $s$-concave. Then we have
$$
\|f\|_\infty\leq C_{n, s}f(\mE X),
$$
whenever $\mE X$ exists, and $
C_{n, s}=\Big(1+\frac{ns}{1+s}\Big)^{1/s}.$
\end{cor}

\begin{proof}
When $s>0$, since $f^s$ is concave, we have 
$$
f^s(\mE X)\geq \mE f^s(X)=\int_{\mR^n} f^{s+1}(x)dx=e^{-sh_{s+1}(f)}.
$$
When $s<0$, since $f^s$ is convex, we have 
$$
f^s(\mE X)\leq \mE f^s(X)=e^{-sh_{s+1}(f)}.
$$
Raising to the power $1/s$ yields, in both cases,
$$
f(\mE X)\geq e^{-h_{s+1}(f)}.
$$
On the other hand, by Corollary~\ref{cor:inf-h}, we have
$$
h_{s+1}(f)\leq h_\infty(f) + h_{s+1}(f_s)-h_\infty(f_s).
$$
Hence, we have
$$
f(\mE X)\geq e^{-h_{\infty}(f)} e^{h_\infty(f_s)-h_{s+1}(f_s)} =\|f\|_\infty e^{h_\infty(f_s)-h_{s+1}(f_s)},
$$
which is the desired bound, and $h_p(f_s)$ are given in \eqref{eq:Lsp1}, \eqref{eq:Lsp2} and \eqref{eq:Lsp3}.
\end{proof}

One can check that Corollary \ref{coro:fra} is sharp for functions defined in \eqref{eq:extr}. Of course, if $s$ is too negative, it will not have finite mean, but the bound is useful in other cases.

\par\vspace{.1in}
\noindent {\bf Acknowledgments:} 
M.F. was supported in part by the Agence Nationale de la Recherche, project GeMeCoD (ANR 2011 BS01 007 01),
J.L. and  M.M. were supported in part by the U.S. National Science Foundation through grants DMS-1409504 and
CCF-1346564.


\end{document}